\documentclass[a4paper, 10pt]{article}
\usepackage[utf8]{inputenc}
\usepackage[T1]{fontenc}
\usepackage{lmodern}
\usepackage[left=3cm,right=3cm,top=2cm,bottom=2cm]{geometry}
\usepackage{amsmath, amssymb}
\usepackage{amsthm}
\usepackage{graphicx}
\usepackage[dvipsnames]{xcolor}
\usepackage{pgf,tikz,tikz-cd}
\usetikzlibrary{arrows}
\usepackage{array}
\usepackage{verbatim}
\usepackage{multicol}
\usepackage{mdframed}
\usepackage{pstricks}
\usepackage{pstricks-add}
\usepackage{tabularx}
\usepackage{fancybox}
\usepackage{stmaryrd}
\usepackage{comment} 
\usepackage{caption}
\usepackage{mathrsfs}
\usepackage{enumitem}
\usepackage[all, cmtip]{xy}
\usepackage{pdfpages}
\usepackage{titlesec}
\usepackage{supertabular}
\usepackage{quiver}

\usepackage{todonotes}
\setuptodonotes{linecolor=blue, bordercolor=blue, backgroundcolor=white, size=footnotesize}

\usepackage{hyperref}
\hypersetup{
    colorlinks=true,
    linkcolor=BlueViolet,
    filecolor=magenta,      
    urlcolor=violet,
    pdftitle={CV Bernard},
    pdfpagemode=FullScreen,
    }

\theoremstyle{remark}

\newtheorem{theo}{test}[section]
\newtheorem{rmq}[theo]{Remark}
\newtheorem{ex}[theo]{Example}
\theoremstyle{plain}

\newtheorem*{rappel}{Theorem}

\theoremstyle{plain}
\usepackage{thmbox}

\newtheorem{theorem}[theo]{Theorem}
\newtheorem{prop}[theo]{Proposition}
\newtheorem{lemma}[theo]{Lemma}
\newtheorem{cor}[theo]{Corollary}
\newtheorem{definition}[theo]{Definition}

\DeclareMathOperator{\Spec}{Spec}

\DeclareMathOperator{\spm}{Spm}

\DeclareMathOperator{\im}{Im}

\DeclareMathOperator{\Dom}{Dom}

\newcommand{\m}{\mathfrak{m}}

\newcommand{\p}{\mathfrak{p}}

\newcommand{\Ocal}{\mathcal{O}}

\newcommand{\Zcal}{\mathcal{Z}}

\newcommand{\Vcal}{\mathcal{V}}
\newcommand{\Dcal}{\mathcal{D}}

\newcommand{\Rad}{\mathrm{Rad}}
\newcommand{\NilRad}{\mathrm{NilRad}}
\newcommand{\inj}{\hookrightarrow}

\newcommand{\KL}{\mathcal{K}^L}
\newcommand{\KO}{\mathcal{K}^0}

\newcommand{\K}{\mathcal{K}}
\newcommand{\piC}{\pi_{_\mathbb{C}}}

\newcommand{\pik}{\pi_k}
\newcommand{\C}{\mathbb{C}}
\newcommand{\R}{\mathbb{R}}
\newcommand{\A}{\mathbb{A}}
\newcommand{\Pbb}{\mathbb{Pbb}}

\title{\Large\textbf{Relative Lipschitz saturation of complex algebraic varieties}}
\date{}
\author{François Bernard}
\begin{document}

\maketitle

\begin{abstract}
    This paper is devoted to the study of the relative Lipschitz saturation of complex algebraic varieties. More precisely, we investigate the concept of Lipschitz saturation of a variety in another, and we focus on the case where the dominant morphism between the two varieties is not necessarily finite. In particular, we answer, in the case of algebraic varieties, an open question of Pham and Teissier concerning the finiteness of the Lipschitz saturation of general algebras. Finally, we use the Lipschitz saturation to provide algebraic criteria for two algebraic varieties to be linked by an algebraic morphism, which is a locally biLipschitz homeomorphism on the closed points of the variety.
\end{abstract}

\section*{Introduction}

\let\thefootnote\relax\footnotetext{2020 \textit{mathematics subject classification.} 14M05, 14B05, 13B22}

In this paper, we use the recent work done in \cite{Bernard2021, BFMQ} about seminormalization and continuous rational functions on algebraic varieties in order to study the relative Lipschitz saturation of algebraic varieties.

The concept of Lipschitz saturation was introduced by Pham and Teissier in 1969 in \cite{PhamTeissier} in the context of complex analytic varieties. While having an algebraic definition, the Lipschitz saturation can be obtained by considering the Lipschitz meromorphic functions defined on the analytic germ we consider. As explained in \cite{FGST2020LipschitzAnAlgebraicApproch}, this object is well understood for plane curves and is used to study biLipschitz equivalence between singularities of analytic curves. More generally, the study of bilipschitz equivalence between singularities is an active domain of research, see for example \cite{Pichon2020IntroductionLipschitzGeo, NeumannPichon2013LipCurves, Fernandes2003TopEquiCurveBilip, FernanJelo2023BilipSpaceCurves, Jelonek2021AlgebraicBilipHomeo}. The Lipschitz saturation itself has been studied, some years after Pham and Teissier, by Lipman in \cite{Lipman1975} where he investigates the algebraic properties of the Lipschitz saturation for general algebras. Very recent works about Lipschitz saturation can also be found in \cite{DaSilvaRibeiro2023, GaffneyDaSilva2023, DuarteFlores2023}.
A closely related object that we will use in our study is the seminormalization. Originally introduced by Andreotti and Norguet \cite{Andre} for analytic varieties, by Traverso \cite{T} for general rings, and by Andreotti and
Bombieri \cite{AndreottiBombieri}  for schemes, the seminormalization $X^+$ of an algebraic variety $X$ is the biggest intermediate variety between $X$ and its normalization, which is bijective to $X$. Very recently, the seminormalization of complex algebraic varieties has been studied in \cite{Bernard2021} by using rational functions that extend continuously, for the Euclidean topology, to the closed points of the variety.

For a ring $R$ and an extension $A\inj B$ of $R$-algebras, the Lipschitz saturation of $A$ in $B$ is given by the ring 
$$ A^L_B := \Big\{ b\in B \mid b\otimes 1 - 1\otimes b \in \overline{I} \Big\} $$
where $\overline{I}$ is the integral closure of the ideal $I \subset B\otimes_R B$ generated by the elements of the form $a\otimes 1 - 1\otimes a$ with $a\in A$.
The Lipschitz saturation of an analytic space $(X,\Ocal_X)$ is obtained by considering, at each point $x$, the Lipschitz saturation $A^L_B$ where $A$ is the $\C$-algebra $\Ocal_{X,x}$ and $B$ is the normalization $A'$ of $A$. The use of the normalization here is natural, since locally Lipschitz functions are, in particular, locally bounded. However, one may be interested in looking at the Lipschitz saturation in something different from the normalization, and so, in the article \cite{PhamTeissier}, Pham and Teissier ask the following question:\smallskip

Let $A\inj B$ be an extension of $\C$-algebras, where $B$ is a subring of the total ring of fractions of $A$. Is the Lipschitz saturation $A^L_B$ a subring of $A'$ ?\smallskip

For varieties, it can be rephrased as: do we get the property of local boundedness directly from the algebraic definition of the Lipschitz saturation? 
 The first result of this paper is to provide a counterexample to this question in the case of algebraic varieties. Thus, if one considers a non-finite dominant morphism $Y\to X$ of algebraic varieties, for example a family of varieties, we don't know if the Lipschitz saturation $\C[X]^L_{\C[Y]}$ will be an $\C$-algebra of finite type. This leads us to introduce the notion of \textit{Lipschitz seminormalization} of $X$ in $Y$, that we denote by $X^{L,+}_Y$ and which is obtained by considering the Lipschitz saturation of the coordinate ring $\C[X]$ in the integral closure $\C[X]'_{\C[Y]}$ of $\C[X]$ in $\C[Y]$. Hence, for any dominant morphism $Y\to X$, we get 
$$ Y\to X'_Y\to X^{L,+}_Y\to X $$
where $X'_Y$ is the relative normalization of $X$ in $Y$. The morphism $X^{L,+}_Y\to X$ is finite and, although the inclusion $\C[X'_Y] \inj \C[X']$ is not true in general, it is also birational. This implies that $\C[X^{L,+}_Y] \inj \C[X']$ and so that the Lipschitz seminormalization is well suited in order to consider the Lipschitz saturation of a variety relatively to another. The second goal of this paper is therefore to conduct a study of the Lipschitz seminormalization. Along the way, we will recover the classical results of the Lipschitz saturation in the category of algebraic varieties.

Finally, the last objective of the article is to give algebraic conditions for two varieties to be linked by a locally biLipschitz algebraic homeomorphism. To do this, we apply the results of \cite{BFMQ} concerning Euclidean or Zariski homeomorphisms between algebraic varieties to the Lipschitz saturation.

The paper is organized as follows: In section \ref{SectionPreli}, because we will deal with Euclidean topology on algebraic varieties, we give some classical lemmas about locally bounded functions on a general topological set. Then we recall the notions of normalization, seminormalization, and saturation for rings and algebraic varieties. In section \ref{SectionPrincipale} we start by recalling the definition of the Lipschitz saturation, and we introduce the "Lipschitz seminormalization" of a variety in another. In Example \ref{ExempleQuestionTeissier}, we show that the Lipschitz seminormalization can be different from the Lipschitz saturation. This answers the question of Pham and Teissier presented above. We end subsection \ref{SubsecLipSatLipSemi} by giving a geometric characterization of the relative Lipschitz saturation. For a dominant morphism $\pi : Y\to X$ of complex affine varieties, the Lipschitz saturation $\C[X]^L_{\C[Y]}$ of $\C[X]$ in $\C[Y]$ is given by 
$$ \C[X]^L_{\C[Y]} = \Bigl\{ p\in \C[Y] \mid \text{locally }|p(y_1)-p(y_2)| \leqslant C\|\piC(y_1)-\piC(y_2)\| \Bigr\} $$
where the norm $\|.\|$ is the one induced by the embedding of $\C[Y]$ in some $\C^n$ and the word "locally" refers to the Euclidean topology. In subsection \ref{SubsecLocLipRational}, we look at rational functions on a variety $X$ that extends to $X(\C)$ as a locally Lipschitz function. The ring of those functions is denoted by $\KL(X(\C))$. For a finite morphism $\pi:Y\to X$, we denote by $\piC$ its restriction to $Y(\C)$ and we look at the induced morphism $f\to f\circ\piC$ between the rings of locally Lipschitz rational functions. We prove that this ring morphism is an isomorphism if and only if $\piC$ is a locally biLipschitz homeomorphism. Moreover, we prove the main theorem of this subsection, which describes the Lipschitz seminormalization of a variety in another in terms of Lipschitz rational functions.
\begin{rappel}[\ref{TheoLipSatEstLipRatio}]
    Let $\pi : Y\to X$ be a dominant morphism of complex affine varieties. Then 
    $$ \C[X^{L,+}_Y] \simeq \KL(X(\C))\times_{\KL(Y(\C))} \C[Y]$$
    where $X^{L,+}_Y$ is the Lipschitz seminormalization of $X$ in $Y$.
\end{rappel}
\noindent In other words, the regular functions of $X^{L,+}_Y$ are given by the locally Lipschitz rational functions that become regular when composed by $\piC$. Note that, by taking $Y=X'$, we get $\C[X^L]\simeq \KL(X(\C))$, which is an algebraic version of Pham-Teissier's theorem, and we get that as in the analytic category, the Lipschitz saturation determines a variety up to birational biLipschitz equivalence. We then use these results to study, in subsection \ref{SubsecLocLipAlgMorphisms}, algebraic morphisms which are locally biLipschitz homeomorphisms on the closed points of the varieties. First, we state the universal property of the Lipschitz seminormalization of a variety $X$ in another variety $Y$: It is the biggest intermediate variety admitting a morphism to $X$ such that its restriction to the closed points is a locally biLipschitz homeomorphism. More precisely, we have 
\begin{rappel}[\ref{TheoPULipSemi}]
    Let $Y\to Z\to X$ be dominant morphisms with $\pi_{Z(\C)}:Z(\C)\to X(\C)$ locally biLipschitz, then $\pi_Z$ factors the morphism of Lipschitz seminormalization $X^{L,+}_Y \to X$.
\end{rappel}

We then discuss the role of the finiteness and birationality hypotheses on the morphisms in the universal property of the Lipschitz saturation. For example, we get that if $X$ is an irreducible algebraic variety, then $X^L$ is the biggest variety admitting a morphism to $X$ which is a locally biLipschitz homeomorphism on $X(\C)$.\\

\textbf{Acknowledgement:} The author is deeply grateful to Goulwen Fichou and Jean-Philippe Monnier for very useful discussions, and to Antoni Rangachev and Bernard Teissier for helping improve this paper. This research was partially supported by Plan d’investissements France 2030, IDEX UP ANR-18-IDEX-0001.

\section{Preliminaries}\label{SectionPreli}

Let $k$ be an algebraically closed field of characteristic zero and $X = \Spec(A)$ be an affine algebraic variety with $A$ a $k$-algebra of finite type. Let $k[X] := A$ denote the coordinate ring of $X$. We have $k[X]\simeq k[x_1,\dots,x_n]/I$ for an ideal $I$ of $k[x_1,\dots,x_n]$ and we will always assume $I$ to be radical. We recall that $X$ is irreducible if and only if $k[X]$ is a domain. A morphism $\pi : Y \to X$ between two affine varieties induces the morphism $\pi^* : k[X] \to k[Y]$ which is injective if and only if $\pi$ is dominant. We say that $\pi$ is of finite type (resp. is finite) if $\pi^*$ makes $k[Y]$ a $k[X]$-algebra of finite type (resp. a finite $k[X]$-module). 

The space $X$ is equipped with the Zariski topology, for which the closed sets are of the form $\Vcal(I) := \{ \p \in \Spec(k[X])\mid I\subset \p \}$ where $I$ is an ideal of $k[X]$. We will denote by $\Dcal(I)$ the complement of $\Vcal(I)$. We define $X(k) := \{ \m \in \spm(k[X])\mid \kappa(\m) = k \}$. Thus, if we write $k[X] = k[x_1,\dots,x_n] / I$, the elements of $X(k)$ can be seen as elements of $\spm(k[x_1,\dots,x_n])$ containing $I$. The Nullstellensatz gives us a Zariski homeomorphism between $X(k)$ and the algebraic set $\Zcal(I) := \{ x\in k^n \mid \forall f \in I\text{ }f(x)=0 \} \subset k^n$. If $k=\C$, we can then consider the strong topology on $X(\C)$ induced by the Euclidean topology of $\C^n$. We note $\pik : Y(k) \to X(k)$ the restriction of $\pi$ to $Y(k)$. We will add the prefix «$Z-$» before a property if it holds for the Zariski topology.

\subsection{Locally bounded rational functions on algebraic varieties}

In this subsection, we give some simple lemmas about locally bounded functions, and we recall the notion of normalization of an algebraic variety as well as its link with locally bounded rational functions.

Let $X$ be an affine variety. The total ring of fractions of $k[X]$ is denoted by $\K(X)$. The ring $\K(X)$ (which is a field when $X$ is irreducible) is also the ring of classes of rational fractions on $X$ and is called the ring of rational functions on $X$. It means that it represents the set of classes of regular functions $f$ on a Z-dense Z-open set $U$ of $X(k)$ with the equivalence relation $(f_1, U_1) \sim (f_2, U_2)$ if and only if $f_1 = f_2$ on $U_1 \cap U_2$. The class of a rational function $f$ admits a rational representation $(f,\Dom(f))$ which is maximal for the inclusion.\\

We now give the preliminary lemmas on locally bounded functions that we will need in the next section. For a topological space $E$ and an element $x$ in $E$, we denote by $\Vcal_E(x)$ the set of neighborhoods of $x$. For a complex affine variety, the notation $\Vcal_{X(\C)}(.)$ will always refer to the Euclidean topology. Let $\mathbb{K}$ be $\R$ or $\C$.

\begin{definition}\label{DefLocBounded}
    Let $E$ be a topological space and let $f : E\to \mathbb{K}$. We say that $f$ is locally bounded if 
    $$ \forall x_0\in E\text{\hspace{0.2cm} } \exists V\in\Vcal_{E}(x_0)\text{\hspace{0.2cm} } \exists C>0\text{\hspace{0.2cm} }\forall x\in V\text{\hspace{0.4cm} }|f(x)|<C.$$
    If $f$ is only defined on a subset of $E$, then we say that it is locally bounded on $E$ if it is locally bounded when extended by $0$ to all $E$.
\end{definition}

\begin{rmq}
    In the case where $f\in K(X)$ is a rational function on $X$, then $f$ is locally bounded on $X(\C)$ if and only if
    $$ \forall x_0\in X(\C)\text{\hspace{0.2cm} } \exists V\in\Vcal_{X(\C)}(x_0)\text{\hspace{0.2cm} } \exists C>0\text{\hspace{0.2cm} }\forall x\in V\cap\Dom(f)\text{\hspace{0.4cm} }|f(x)|<C.$$
\end{rmq}

The three following lemmas will be used later to say that a rational function is locally bounded if and only if it is locally bounded when composed by a finite morphism of algebraic varieties.

\begin{lemma}\label{LemLocBoundCircPIisLocBound}
    Let $\pi : F\to E$ be a continuous function between two topological spaces and let $f : E\to \mathbb{K}$ be a locally bounded function. Then $f\circ\pi$ is locally bounded on $F$.
\end{lemma}

\begin{proof}
    Let $y\in F$, we can consider $U\in \Vcal_{E}(\pi(y))$ and $C>0$ such that $|f_{|U}|<C$. Since $\pi$ is continuous, the set $V:=\pi^{-1}(U)$ is a neighborhood of $y$ in $F$ and we have $|f\circ\pi_{|V}|=|f_{|U}|<C$. So $f\circ\pi$ is locally bounded.
\end{proof}

\begin{lemma}\label{LemLocBoundOnCompact}
    Let $E$ be a topological space such that every point admits a compact neighborhood, and let $f:E\to \mathbb{K}$ be a function. Then $f$ is locally bounded if and only if $f$ is bounded on every compact set.   
\end{lemma}

\begin{proof}
    Let $f$ be bounded on every compact set and let $x\in E$. By assumption, we can consider a compact neighborhood of $x$. Then $f$ is bounded on this neighborhood. Conversely, suppose $f$ is locally bounded and let $K$ be a compact set of $E$. Since $f$ is locally bounded, we have that, for all $x\in K$, there exists $V_x\in \Vcal_E(x)$ and $C_x\in \R$ such that $|f_{|V_x}|<C_x$. Since $K$ is compact, we can consider $x_1,\dots,x_n\in K$ such that $$K=\bigcup_{i=1}^n V_{x_i}$$
    Then we get $|f_{|K}|<\max_i C_{x_i}$
\end{proof}

\begin{lemma}\label{LemLocBoundIfCircPIisLocBound}
    Let $\pi : F\to E$ be a proper and surjective map between two topological spaces such that every point of $E$ admits a compact neighborhood, and let $f:E\to \mathbb{K}$ be such that $f\circ\pi$ is locally bounded. Then $f$ is locally bounded.
\end{lemma}

\begin{proof}
    Let $K$ be a compact set of $E$. Then $\pi^{-1}(K)$ is compact because $\pi$ is proper. Then, by hypothesis, there exists $C\in \R$ such that $|f\circ\pi_{|\pi^{-1}(K)}|<C$. Since $\pi$ is surjective, then for all $x\in K$, we can consider $y\in \pi^{-1}(x)$. So, for all $x\in K$, we get $|f(x)|=|f\circ\pi(y)|<C$ which means that $f$ is bounded on $K$. We have shown that $f$ is bounded on every compact set of $E$ so, by lemma \ref{LemLocBoundOnCompact}, $f$ is locally bounded on $E$.
\end{proof}

It is not surprising that, for rational functions defined on algebraic varieties, the property of being locally bounded behaves well by composition with a finite morphism because those functions are deeply connected with the normalization of the variety. Before specifying this connection, let us recall the notion of normalization of a variety and, thus, the notion of integral closure of a ring.

\begin{definition}
Let $A\inj B$ be an extension of rings.
\begin{enumerate}
    \item An element $b\in B$ is \textit{integral} over $A$ if there exists a monic polynomial $P\in A[X]$ such that $P(b)=0$.
    \item We call \textit{integral closure} of $A$ in $B$ and we write $A'_B$ the ring defined by $$A'_B := \Bigl\{ b\in B \mid b \text{ integral on }A\Bigr\}.$$
\end{enumerate}
\end{definition}

For an affine variety $X$, the integral closure of $\C[X]$ in $\K(X)$ is a finite $\C[X]$-module (see \cite{Eisen} Thm 4.14), so it is a $\C$-algebra of finite type. Thus, we can define the \textit{normalization} $X'$ of $X$ such that $X' = \Spec(\C[X]_{\K(X)}')$. We get a finite and birational morphism $\pi' : X' \to X$.
The normalization of $X$ is in fact the biggest affine variety finitely birational to $X$. It means that for every finite, birational morphism $\varphi : Y\to X$, there exists $\psi : X' \to Y$ such that $\pi' = \varphi \circ \psi$. Its link with locally bounded rational functions is given by the following classical proposition.

\begin{prop}[\cite{Lojasiewicz2013}]\label{PropLocalBoundIffIntegral}
    Let $X$ be an affine complex algebraic variety and let $f\in \K(X)$. Then the following properties are equivalent:
    \begin{enumerate}
        \item The function $f$ is locally bounded.
        \item The function $f$ is integral over $\C[X]$.
    \end{enumerate}
\end{prop}

In other words, the regular functions of $X'$ correspond to the locally bounded rational functions on $X$.

\begin{ex}
    Let $X := \Spec(\C[x,y]/\langle y^2-x^2(x+1) \rangle)$ and let $f = y/x \in K(X)$. We have $f^2 -x-1=0$ if $x\neq 0$, thus $f$ is integral over $\C[X]$. Moreover $f$ is regular on $\Dcal(x)$ and, since $|f|^2 = |x+1|$, it is locally bounded when $x$ goes to $0$. The normalization of $X$ is given by $\pi' : X' \to X$ with $X' = \Spec(\C[x,y]/\langle y^2-x-1\rangle)$ and $\piC' : (x,y) \mapsto (x,xy)$. Hence $f\circ \piC' = y$ is indeed a polynomial function on $X'(\C)$.
\end{ex}

\subsection{Seminormalization and saturation}

Our study of the relative Lipschitz saturation for algebraic varieties will often use the notions of seminormalization and saturation that we present in this section. For a more complete presentation, we refer to Vitulli's survey \cite{VitulliSurvey2011} from 2011 about seminormalization. These notions have been studied recently in \cite{Bernard2021} and \cite{BFMQ} using continuous rational functions on complex algebraic varieties.

We start by recalling the definitions of seminormalization and saturation for general rings. 

\begin{definition}
Let $A \inj B$ be an extension of rings. We define the seminormalization of $A$ in $B$ as

$$A_B^{+} := \Bigr\{b \in A'_B \mid \forall \p \in \Spec(A) \text{, } b_{\p} \in A_{\p}+\Rad(B_{\p})\Bigr\}$$

\noindent where $\Rad(B_{\p})$ is the Jacobson radical of $B_{\p}$. The ring $A_B^+$ is the \textit{seminormalization} of $A$ in $B$. If $A = A_B^+$, then $A$ is said to be \textit{seminormal} in $B$.
\end{definition}

\begin{definition}
Let $A \inj B$ be an extension of rings. We define the saturation of $A$ in $B$ as

$$\widehat{A_B} := \Bigr\{b \in B \mid b\otimes_A 1 - 1\otimes_A b \in \NilRad(B\otimes_A B) \Bigr\}$$

where the nilradical $\NilRad$ denotes the ideal of nilpotent elements. The ring $\widehat{A_B}$ is the \textit{saturation} of $A$ in $B$. If $A = \widehat{A_B}$, then $A$ is said to be \textit{saturated} in $B$.
\end{definition}

\begin{rmq}
    Let $R$ be a ring. If $A$ and $B$ are $R$-algebras, then 
    $$\widehat{A_B} = \Bigr\{b \in B \mid b\otimes_R 1 - 1\otimes_R b \in \sqrt{\ker \phi_B} \Bigr\}$$
    where $\phi_B : B\otimes_{R} B \to B\otimes_{A}B$ send $b_1\otimes_{R}b_2$ to $b_1\otimes_A b_2$. We will tell more about this kernel when we will give the definition of the Lipschitz saturation which is closely related to this definition of the saturation.
\end{rmq}

\begin{rmq}\label{RmqSaturationEtSemiCorrespondentParfois}
    For rings having characteristic zero, we have $A^+_B = \widehat{A_{A'_B}}$. In particular, if the extension $A\inj B$ is integral, then $A_B^{+} = \widehat{A_B}$. See for example \cite{BFMQ} Proposition 3.2. 
\end{rmq}

We recall now the universal property of the seminormalization of rings using the notion of "subintegral extensions" of rings.

\begin{definition}
    An extension of rings $A\inj B$ is called \textit{subintegral} if it is integral and if the induced map $\Spec(B)\to\Spec(A)$ is bijective and equiresidual (it gives isomorphisms between the residue fields).
\end{definition}

The property of being subintegral is transitive.

\begin{prop}[\cite{T} Lemma 1.2]\label{PropSousExtSubEstSub}
Let  $A \inj C \inj B$ be integral extensions of rings. Then the following properties are equivalent 
\begin{enumerate}
    \item[1)] The extension $A\inj B$ is subintegral.
    \item[2)] The extensions $A\inj C$ and $C\inj B$ are subintegral.
\end{enumerate}
\end{prop}

The universal property of the seminormalization states that $A^+_B$ is the biggest subintegral subextension of $A$ in $B$.

\begin{theorem}[\cite{T} Universal property of seminormalization]\label{TheoPUSNanneaux}
Let $A\inj B$ be an integral extension of rings. Then $A\inj A^+_B \inj B$ is the unique subintegral subextension such that, for every intermediate subintegral subextension $A\inj C\inj B$, the image of $C$ by the injection $C\inj B$ is contained in $A_B^+$.
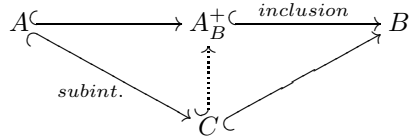
$$\xymatrix{
   A \ar@{_{(}->}[rrd]_{subint.} \ar@{^{(}->}[rr]& & A_B^+ \ar@{^{(}->}[rr]^{inclusion} & & B \\
    && C \ar@{^{(}.>}[u] \ar@{_{(}->}[urr] &&
}$$
\end{theorem}

For algebraic varieties, the seminormalization of the coordinate ring of a variety in another produces a new variety called the "relative seminormalization". We could also consider the saturation of a coordinate ring in another, but we don't know if it always produces a new algebraic variety.

\begin{prop}[\cite{BFMQ} Proposition 4.10]\label{PropConstantFibres}
    Let $\pi : Y\to X$ be a dominant morphism of complex affine varieties. Then the coordinate ring of the seminormalization of $X$ in $Y$ is given by 
    $$ \C[X^+_Y] := \C[X]^+_{\C[Y]} = \Bigr\{p \in \C[X'_Y] \mid p \text{ is constant on the fibers of }\pik' \Bigr\} $$
    where $\pik' : X'_Y \to X$ is relative normalization morphism.
\end{prop}

The seminormalization $X^+$ of a variety $X$ is the variety obtained by taking $Y=X'$ in the previous proposition. In \cite{Bernard2021}, it is shown that the seminormalization is deeply related to rational continuous functions defined on the closed points of the variety.

\begin{definition}
    Let $X$ be a complex affine variety. We denote by $\KO(X(\C))$ the set of rational functions of $\K(X)$ that extend continuously to $X(\C)$ for the Euclidean topology.
\end{definition}

\begin{ex}
Let $X := \Spec(\C[x,y]/\langle y^2-x^3\rangle)$. The rational function $y/x$ belongs to $\KO(X(\C))$ because it extends continuously by $0$ in $(0,0)$ since $|y/x| = |x|^{1/2} \to 0$ when $(x,y) \to 0$.
\end{ex}

We can reinterpret subintegral extensions of coordinate rings with continuous rational functions or with topological conditions on $\piC$.

\begin{prop}[\cite{Bernard2021} Proposition 4.12]\label{PropSubEquiMemeFctRatioCont}
Let $\pi : Y \to X$ be a finite morphism of affine varieties. Then the following properties are equivalent:
\begin{enumerate}
    \item[1)] The extension $\pi^* : \C[X] \inj \C[Y]$ is subintegral.
    \item[2)] The morphism $\pi^* : \KO(X(\C)) \to \KO(Y(\C))$ is an isomorphism.
    \item[3)] The morphism $\piC$ is a bijection.
    \item[4)] The morphism $\piC$ is a homeomorphism for the Euclidean topology.
\end{enumerate}
\end{prop}

Finally, we state the link between relative seminormalization and continuous rational functions.

\begin{theorem}[\cite{BFMQ} Theorem 4.11]\label{TheoIsoSchemaSN}
Let $\pi : Y\to X$ be a dominant morphism of complex affine varieties. Then 
$$ \C[X^+_Y] \simeq \KO(X(\C))\times_{\KO(Y(\C))} \C[Y]  $$
where the right-hand side stands for the fiber product of the rings.
\end{theorem}

By taking $Y=X'$, we get the main result of \cite{Bernard2021} which says that the seminormalization is obtained by replacing the structural sheaf of a variety by the sheaf of continuous rational functions defined on its closed points.

\section{Lipschitz saturation and locally Lipschitz rational functions}\label{SectionPrincipale}

As said in the preliminaries, rational functions that are locally bounded correspond to the regular functions on the normalization, and rational functions that are continuous for the Euclidean topology correspond to the regular functions on the seminormalization. In the same spirit, we look at rational functions that are locally Lipschitz and which correspond, as in the analytic case developed in \cite{PhamTeissier}, to the regular functions of a variety called the "Lipschitz saturation". 
We will start by studying, in subsection \ref{SubsecLipSatLipSemi}, the Lipschitz saturation of an algebraic variety in another with consideration for the finiteness of this construction. Then, in subsection \ref{SubsecLocLipRational}, we look at locally Lipschitz rational functions and their link to the relative Lipschitz saturation. Finally, in subsection \ref{SubsecLocLipAlgMorphisms}, we will state some algebraic criteria for two varieties to be linked by a locally biLipschitz algebraic morphism.

\subsection{Lipschitz saturation and Lipschitz seminormalization}\label{SubsecLipSatLipSemi}

We begin by recalling the definition and some classical properties of the Lipschitz saturation. Then we will investigate questions about the finiteness of this construction. This will lead us to define the \textit{Lipschitz seminormalization} of a variety, and, by proving that it can be different from the Lipschitz saturation, we will answer "Question 1" of \cite{PhamTeissier} in the case of algebraic varieties. Finally, we use the classical geometric interpretation of the integral closure of an ideal in order to give a description of the Lipschitz saturation of a variety in another. This is the Lipschitz analog of Proposition \ref{PropConstantFibres} about seminormalization.

Let us start by recalling the definition of the Lipschitz saturation and, thus, the definition of the integral closure of an ideal.

\begin{definition}
    Let $A$ be a ring and $I\subset A$ an ideal. The integral closure $\overline{I}$ of $I$ is 
    $$ \overline{I} := \left\{ b\in A \mid \exists a_i \in I^i\text{, } b^n+\sum_{i=1}^n a_ib^{n-i} = 0 \right\}.$$
\end{definition}

One can look at \cite{SwansonHuneke} for more information about integral closure of ideals. We now give the definition from \cite{PhamTeissier} of Lipschitz saturation.

\begin{definition}
    Let $A \inj B$ be an extension of $\C$-algebras. The \textit{Lipschitz saturation} $A^L_B$ of $A$ in $B$ is defined by
    $$ A^L_B := \left\{ b\in B \mid b\otimes 1 - 1\otimes b \in \overline{\ker\phi_B} \right\} $$
    where $\phi_B : B\otimes_{\C} B \to B\otimes_{A}B$ send $b_1\otimes_{\C}b_2$ to $b_1\otimes_A b_2$.
    
    We say that the extension $A\inj B$ is \textit{Lipschitz saturated} if $A^L_B = B$.
\end{definition}

An extensive study of the algebraic properties of the Lipschitz saturation is made by Lipman in \cite{Lipman1975}. When $B$ is not the normalization of $A$, we will sometimes call $A^{L}_B$ the \textit{relative} Lipschitz saturation. It is not really the same as in Lipman where the word "relative" highlights the fact that he considers the Lipschitz saturation of $R$-algebras, where $R$ is any ring.

\begin{rmq}
{\color{white} blank}
    \begin{itemize}
        \item[(1)] The ideal $\ker\phi_B$ of $B\otimes_{\C} B$ is generated by the elements of the form $a \otimes 1- 1\otimes a$ with $a\in A$.
        \item[(2)] For all $p,q\in B$, we have $$(p\otimes 1- 1\otimes p)(q\otimes 1)+(q\otimes 1- 1\otimes q)(1\otimes p) = pq\otimes 1 - 1\otimes pq $$ and $$ (p\otimes 1- 1\otimes p) + (q\otimes 1 - 1\otimes q) = (p+q)\otimes 1- 1\otimes (p+q) $$
    This shows that $A^L_B$ is a $\C$-algebra. Moreover, if $\pi : Y \to X$ is a dominant morphism between complex affine varieties, then $\ker\phi_{\C[Y]}$ is generated by the elements $x_i\circ \pi\otimes 1 - 1\otimes x_i\circ\pi$ where the $x_i$ are the coordinates of $X$. See \cite{DaSilva2023survey} Example 1.6 for more details.
        \item[(3)] For all ideal $I$, we have $I\subseteq \overline{I} \subseteq \sqrt{I}$. This implies that the Lipschitz saturation $A^L_B$ is a subalgebra of the saturation $\widehat{A_B}$. 
        \item[(4)] If $\pi : Y\to X$ is a finite morphism between complex affine varieties, i.e. $\C[Y]$ is a finite $\C[X]$-module, then $\C[X]^L_{\C[Y]}$ is also a finite $\C[X]$-module and so it is $\C$-algebra of finite type. This means that we can consider the variety $X^L_Y = \Spec(\C[X]^L_{\C[Y]})$. In the case where $Y$ is the normalization of $X$, this variety corresponds to the algebraic version of the classical Lipschitz saturation that appears, for example, in \cite{Adkins}.
    \end{itemize}

\begin{ex}
    Let $X = \Spec(\C[x,y]/\langle xy(y-x) \rangle)$ and $f=\frac{2xy}{x+y}$ be a rational function on $X$.
    Let us see that $f$ belongs to $\C[X]^L$. In order to do so, we have to see that $f\otimes 1- 1\otimes f \in \overline{I}$, where $I$ is the ideal of $\C[X]'\otimes_{\C}\C[X]'$ generated by the elements of the form $p\otimes 1- 1\otimes p$ with $p\in\C[X]$. Let
    $$g = \frac{2xy}{x^2+y^2}\otimes 1-1\otimes \frac{2xy}{x^2+y^2}.$$
    Since $\left( \frac{2xy}{x^2+y^2} \right)^2 = \frac{2xy}{x^2+y^2}$, we have $\frac{2xy}{x^2+y^2} \in \C[X]'$. So we get $g\in \C[X]'\otimes_{\C}\C[X]'$. Let $a = g(y\otimes 1 - 1\otimes y)+(x\otimes 1 - 1\otimes x)\in I$
    and $b = g(x\otimes 1 - 1\otimes x)(y\otimes 1 - 1\otimes y) \in I^2$. One can check, by looking at each irreducible component of $X\times X$, that 
    $$ (f\otimes 1 - 1\otimes f)^2 + a(f\otimes 1 - 1\otimes f) + b = 0.$$
    Hence $f \in \C[X]^L$. 
\end{ex}
    
    In the case where $\pi : Y\to X$ is not a finite morphism, we don't know, as for $\widehat{\C[X]_{\C[Y]}}$, if $\C[X]^L_{\C[Y]}$ corresponds to the coordinate ring of an affine variety.
    These considerations about the finiteness of morphisms were studied in \cite{BFMQ} about saturation and seminormalization. This leads us to define the notion of "Lipschitz seminormalization" of a variety.
\end{rmq}

\begin{definition}
    Let $A\inj B$ be an extension of $\C$-algebras. The \textit{Lipschitz seminormalization} $A^{L,+}_B$ of $A$ in $B$ is defined by
    $$ A^{L,+}_B := \left\{ b\in A'_B \mid b\otimes 1 - 1\otimes b \in \overline{\ker\phi_B} \right\} $$
    where $\phi_B : B\otimes_{\C} B \to B\otimes_{A}B$.
\end{definition}

\begin{rmq}
    The Lipschitz seminormalization $A^{L,+}_B$ is a subalgebra of the seminormalization $A^+_B$, see remark \ref{RmqSaturationEtSemiCorrespondentParfois}.
\end{rmq}

If $\pi : Y\to X$ is a dominant morphism of complex affine varieties, then by \cite{StackProj} Lemma 29.53.15, we can consider the relative normalization of $X$ in $Y$ which comes with a finite morphism $X'_Y\to X$. Then the inclusions
$$\C[X]\inj \C[X]^{L,+}_{\C[Y]} \inj \C[X'_Y]$$
imply that we can define the Lipschitz seminormalization $X^{L,+}_Y$ of $X$ in $Y$ and that it comes with a finite morphism $X^{L,+}_Y \to X$. If $\pi : Y\to X$ is a finite morphism, then $\C[X]\inj \C[Y]$ is an integral extension of rings, i.e. $\C[X'_Y] = \C[Y]$. Then $X^{L,+}_Y = X^L_Y$. In particular, this is true if $Y$ is the normalization of $X$. This means that, for any complex affine variety $X$, the Lipschitz seminormalization of $X$ is equal to the Lipschitz saturation of $X$ which we denote by $X^L$.\\

To summarize all the definitions, when $\pi : Y\to X$ is not finite, we have the following diagram

$$\xymatrix{
        & \C[X^L] \ar@{^{(}->}[r] & \C[X^+] \ar@{^{(}->}[r] & \C[X'] \ar@{^{(}->}[r] & \K(X) \ar@{^{(}->}[dr] & \\
        \C[X] \ar@{^{(}->}[r] \ar@{^{(}->}[dr] & \C[X^{L,+}_Y] \ar@{^{(}->}[r] \ar@{^{(}->}[u] \ar@{_{(}->}[d] & \C[X^+_Y] \ar@{^{(}->}[r] \ar@{^{(}->}[u] \ar@{_{(}->}[d] & \C[X'_Y] \ar@{^{(}->}[r] & \C[Y] \ar@{^{(}->}[r]& \K(Y)\\
        & \C[X]^L_Y \ar@{^{(}->}[r] & \widehat{\C[X]_{\C[Y]}} \ar@{^{(}->}[urr] &&
    }$$
    
\noindent The inclusion $\C[X^+_Y] \subset \C[X^+]$ can be seen with \cite{BFMQ} Theorem 4.11 since $\C[X^+] \simeq \KO(X(\C))$. This implies $\C[X^{L,+}_Y] \inj \C[X^+]\inj \C[X']$ and so $\C[X^{L,+}_Y]\inj \C[X^L]$.

\begin{ex}\label{ExempleQuestionTeissier}
    In the article \cite{PhamTeissier}, it is asked as "Question 1" if, for general algebras $A$ and $B$ such that $A\inj B\inj \K(A)$ where $\K(A)$ denotes the total fractions ring of $A$, the relative Lipschitz saturation $A^L_B$ is a subalgebra of $A'_{\K(A)}$. Considering our new definitions, it is the same as asking if the Lipschitz saturation and the Lipschitz seminormalization are always equal. We bring a negative answer to this question, in the case of complex algebraic varieties, by using example 3.15 from \cite{BFMQ}.

    Let $X = \Spec\left( \C[x,y]/\langle y^2-x^2(x+1)\rangle \right)$. Its normalization $X'$ is obtained by adding the element $y/x$ to its coordinate ring. So $\C[X'] = \C[x,y]/\langle y^2-(x+1) \rangle$ and $\pi' : X'\to X$ is given by $(x,y)\mapsto (x,xy)$. Let $Y$ be defined by removing the point $(0,1)$ of $X'$ lying over the singular point of $X$. Its coordinate ring is given by $\C[Y] = \C[X'][\frac{1}{y-1}] = \C[x,y,z]/\langle y^2-(x+1),z(y-1)-1\rangle $ and the induced morphism $\pi : Y\to X$ is given by $(x,y,z)\to (x,xy)$. An important point of this example is that \textbf{the morphism $\pi : Y\to X$ is not finite}.

    \begin{figure}[h]
        \centering
    \begin{tikzpicture}[line cap=round,line join=round,>=triangle 45,x=1.2cm,y=1.2cm]
    \clip(-7,-1.2) rectangle (5,1.2);
    
    \draw [line width=1pt] (-4.5,0) -- (-3,0);
    \draw [line width=1pt] (-4.5,0) -- (-4.4,0.1);
    \draw [line width=1pt] (-4.5,0) -- (-4.4,-0.1);
    \draw [line width=1pt] (0,0) -- (1.5,0);
    \draw [line width=1pt] (0,0) -- (0.1,0.1);
    \draw [line width=1pt] (0,0) -- (0.1,-0.1);
    
    \draw[line width=1pt,smooth,samples=100,domain=-1:1, xshift = -7cm] plot(\x,{sqrt((\x)^(2)*((\x)+1))});
    \draw[line width=1pt,smooth,samples=100,domain=-1:1, xshift = -7cm] plot(\x,{0-sqrt((\x)^(2)*((\x)+1))});
    
    \draw[line width=1pt,smooth,samples=100,domain=-1:1, xshift = -1.5cm] plot(\x,{sqrt((\x)+1)/1.5});
    \draw[line width=1pt,smooth,samples=100,domain=-1:1, xshift = -1.5cm] plot(\x,{(-sqrt((\x)+1))/1.5});
    
    \draw[line width=1pt,smooth,samples=100,domain=-1:1, xshift = 4cm] plot(\x,{sqrt((\x)+1)/1.5});
    \draw[line width=1pt,smooth,samples=100,domain=-1:1, xshift = 4cm] plot(\x,{(-sqrt((\x)+1))/1.5});
    \draw[fill=white] (3.15,0.6) circle (3.5pt);
    
    \draw (-5,-0.8) node[anchor=north west] {$X$};
    \draw (-0.1,-0.8) node[anchor=north west] {$X'$};
    \draw (4.5,-0.8) node[anchor=north west] {$Y$};
    
    \end{tikzpicture}
        \caption{Lipschitz saturation different from Lipschitz seminormalization}
        \label{fig:LipSatDiffLipSN}
\end{figure}
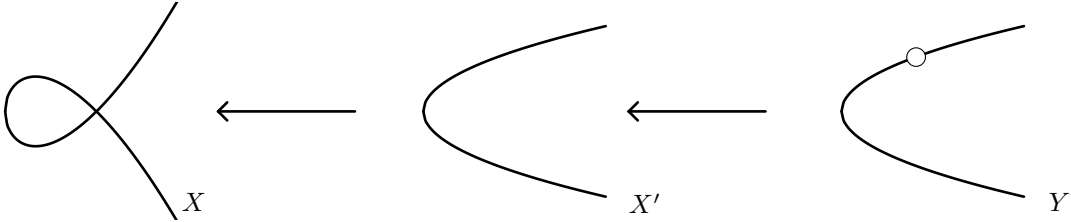

    The variety $X$ is seminormal because it has multicross singularities (see \cite{Multicross}). So $X=X^+$ and since $\C[X]\inj \C[X^{L,+}_Y]\inj \C[X^+]$, this implies that $\C[X] = \C[X^{L,+}_Y]$. To show that $\C[X^L_Y] = \C[Y]$, we do the same computations as in \cite{BFMQ} Example 3.15. For clarity of the exposition, we present these computations here. Let $\alpha = y\otimes 1-1\otimes y$ and $\beta = z\otimes 1-1\otimes z$. We want to show that $\alpha = \beta = 0$ in $\C[Y]\otimes_{\C[X]}\C[Y]$. First, see that $x = y^2-1$ implies $zx =y+1$ in $\C[Y]$. So $x\beta = (y+1)\otimes 1-1\otimes (y+1) = \alpha$. Moreover $x\alpha = xy\otimes 1-1\otimes xy = (y\circ\pi) \otimes 1 - 1\otimes (y\circ\pi) = 0$. Then $\alpha^2 = \alpha x\beta = 0$ and since we can show, by straightforward computation, that $\alpha = \frac{1}{4}\alpha^3$, we get $\alpha = 0$. Finally, using the equality $\alpha = (y-1)\otimes 1 - 1\otimes (y-1)$ and the relation $z(y-1)=1$, we observe that $0 = (z\otimes z)\alpha = -\beta$. We have shown that $\alpha$ and $\beta$ belong to the ideal $I = \langle p\otimes 1- 1\otimes p \mid p\in \C[X]\rangle$, hence $y,z\in \C[X]^L_{\C[Y]}$. This means that $\C[X^L_Y] = \C[Y]$.
    
    In conclusion, we have $X^{L,+}_Y = X \neq Y = X^L_Y$ even though the curves $X$ and $Y$ are irreducible, birational, and bijective. Again, the key point is that the morphism $Y\to X$ is not finite.
    
    For examples in any dimension, we consider Example 3.6 of \cite{BFMQ} and we show that, for $n\geqslant 1$, the variety $X\times \A^n$ is Lipschitz seminormal in $Y\times \A^n$ and the Lipschitz saturation of $X\times \A^n$ in $Y\times \A^n$ is $Y\times \A^n$. By \cite{GrecoTraverso80} Corollary 5.9, the variety $X\times \A^n$ is seminormal because $X$ and $\A^n$ are seminormal, so $X\times \A^n$ is, in particular, Lipschitz seminormal in $Y\times \A^n$. Moreover, since $\C[Y\times \A^n] \simeq \C[Y][x_1,\dots,x_n]$, the computations done previously still prove that $\C[X\times \A^n]^L_{\C[Y\times \A^n]}$ contains $y$ and $z$, so it is equal to $\C[Y\times \A^n]$.
\end{ex}

\begin{rmq}
Let $A\inj B$ be an extension of $\C$-algebras and let $J$ be an ideal of $B\otimes_{\C} B$ containing $\ker \phi_B$. Then one could define the following saturation:
$$\Big\{ b\in B \mid b\otimes 1- 1\otimes b\in J\Big\}$$
which will correspond to the Lipschitz saturation if $J = \overline{\ker \phi_B}$ and to the seminormalization if $J = \sqrt{\ker \phi_B}$. The point of this remark is that, since the previous calculations proved $\alpha$ and $\beta$ to be elements of $\ker \phi_{\C[Y]}$ itself, Example \ref{ExempleQuestionTeissier} will work for any such $J$.
\end{rmq}

We now want to give a more geometric meaning to the definition of the Lipschitz saturation of an algebraic variety in another one. More precisely, for a dominant morphism $\pi: Y\to X$, we identify the polynomial functions of $\C[Y]$ that come from an element of $\C[X]^L_{\C[Y]}$. The following proofs are classical, we include them here in order to verify that they also work for algebraic varieties equipped with the Euclidean topology on their closed points, and also to verify that Proposition \ref{PropLipSatCritereGeo} is true for non-finite morphisms. We start with the following lemma, whose proof can be found, for example, in \cite{Hochster}.

\begin{lemma}\label{LemHochster}
    Let $X$ be a complex affine variety and $I = \langle p_1,\dots,p_n\rangle$ be an ideal of $\C[X]$. Let 
    $$S_i := \C[X]\left [ \frac{p_1}{p_i},\dots,\frac{p_n}{p_i} \right ] \subseteq \C[X]_{p_i} $$
    and $T_i = S_i'$. Then $p\in \overline{I}$ if and only if $p\in I.T_i$ for all $i$.
\end{lemma}

\begin{proof}
    For all $i\in \{1,\dots,n\}$, the ideal $I.T_i$ is generated by $p_i$. So, if $p\in\overline{I}$, we divide its integral relation $p^d+a_1.p^{d-1}+\dots+a_d=0$ by $p_i^d$ and since, for all $k\in \{1,\dots,d\}$, we have $p_i^k\mid a_k$ in $S_i$, we get that $p/p_i$ is integral over $S_i$. Then $p/p_i\in T_i$ and $p=\frac{p}{p_i}p_i \in I.T_i$.
    For the reverse implication, we use \cite{SwansonHuneke} Proposition 6.8.3, which states that we can assume $X$ to be irreducible and we have
    $$\overline{I} = \bigcap_{V} I.V \cap \C[X]$$
    where $V$ varies over all valuation domains such that $\C[X]\subset V\subset \K(X)$. Let $p\in \C[X]$ be such that $p\in I.T_i$, for all $i\in \{1,\dots,n\}$ and let $\C[X] \subset V \subset \K(X)$ be a valuation domain. Then the elements of $V$ are totally ordered by divisibility, which implies that the image of one of the $p_i$ divides all the others. Then, for this specific $i$, we have $S_i\subseteq V$ and since $V$ is normal, we obtain $T_i\subseteq V$. Finally $p\in I.T_i$ implies $p\in I.V$.
\end{proof}

Thanks to Lemma \ref{LemHochster}, we get the next proposition, which gives a geometric meaning to the condition for a polynomial function on an affine variety to belong to the integral closure of an ideal. The proof follows the one of Pham and Teissier, and this result can also be found in \cite{LejeuneTeissierCloture} as Theorem 2.1.

\begin{prop}\label{PropEclatementNormalise}
    Let $X$ be a complex affine variety and $I$ be an ideal of $\C[X]$ with $I=\langle p_0,\dots,p_s \rangle$. Let $p\in \C[X]$. Then the following properties are equivalent:
    \begin{enumerate}
        \item The element $p$ belongs to $\overline{I}$.
        \item We have $|p(x)|<C\sup_i |p_i(x)|$ locally on $X(\C)$. 
    \end{enumerate}
\end{prop}

\begin{rmq}
    The word "locally" in the second property is a short way to say
    $$ \forall x_0\in X(\C)\text{\hspace{0.2cm} } \exists V\in\Vcal_{X(\C)}(x_0)\text{\hspace{0.2cm} } \exists C>0\text{\hspace{0.2cm} }\forall x\in V\text{\hspace{0.4cm} }|p(x)|<C\sup_i |p_i(x)|$$
\end{rmq}

\begin{proof}
    Let us consider $$ \begin{array}{cccc}
        \psi : & X\setminus \Zcal(I) & \to & \Pbb^r \\
         & x & \mapsto & [p_0(x):\dots:p_r(x)]
    \end{array} $$
    and $Z := \overline{\Gamma_{\psi}}^{\Zcal}\subset X\times \Pbb^r$. Then the morphism $\pi : Z\to X$ is the blow-up of $X$ along $\Zcal(I)$. The elements of $Z(\C)$ are of the form 
    $$ \Bigl( x, [u_1:\dots:u_r] \Bigr) \in X(\C)\times \Pbb^r(\C)\text{ with }p_i(x)u_j=p_j(x)u_i$$
    We can consider the charts $U_i(\C) := \{u_i=1\}$ such that $Z(\C) = \bigcup U_i(\C)$ and, for all $i\in \{1,\dots,r\}$, the ideal $I.\Ocal_Z(U_i)$ is generated by $p_i\circ\pi$. The property $|p(x)|<C\sup_i |p_i(x)|$ locally on $X(\C)$ is equivalent to $\frac{|p|}{\sup_j|p_j|}$ locally bounded on $X(\C)$. By Lemma \ref{LemLocBoundCircPIisLocBound} and \ref{LemLocBoundIfCircPIisLocBound}, this is equivalent to $\frac{|p\circ\pi|}{\sup_j|p_j\circ\pi|}$ locally bounded on $Z(\C)$ because $\piC$ is surjective, proper, and continuous for the Euclidean topology. Since, on every $U_i(\C)$, we have $p_j\circ\pi=u_j.(p_i\circ\pi)$ for all $j\in \{1,\dots,s\}$, then $|p_j\circ\pi|=|u_j|.|p_i\circ\pi|< \Tilde{C}.|p_i\circ\pi|$ locally on $U_i(\C)$. So, having $|p\circ\pi|<C\sup_j|p_j\circ\pi|$ locally on $Z(\C)$ is equivalent to having, for all $i\in \{1,\dots,s\}$, $|p\circ\pi|<C|p_i\circ\pi|$ locally on $U_i(\C)$. Then $\frac{|p\circ\pi|}{\sup_j|p_j\circ\pi(x)|}$ is locally bounded on $Z(\C)$ if and only if $\frac{p\circ\pi}{p_i\circ\pi}$ is locally bounded on $U_i(\C)$, for all $i\in\{1,\dots,s\}$. By Lemma \ref{PropLocalBoundIffIntegral}, this is equivalent to have $\frac{p\circ\pi}{p_i\circ\pi}\in \Ocal_{Z'}(U_i)$. Since $I.\Ocal_{Z}(U_i)$ is generated by $p_i\circ\pi$, this is equivalent to $p\circ\pi \in I.\Ocal_{Z'}(U_i)$, for all $i$. Finally, by Lemma \ref{LemHochster}, we get $$\text{for all $i$, \hspace{0.2cm}} p\circ\pi\in I.\Ocal_{Z'}(U_i) \iff p\in \overline{I}$$
\end{proof}

We can now apply this geometric interpretation of the integral closure of an ideal to the definition of the Lipschitz saturation. Note that the following proposition does not require the morphism to be finite.

\begin{prop}\label{PropLipSatCritereGeo}
    Let $\pi : Y\to X$ be a dominant morphism of complex affine varieties. Then
    $$ \C[X]^L_{\C[Y]} = \Bigl\{ p\in \C[Y] \mid \text{locally }|p(y_1)-p(y_2)| \leqslant C\|\piC(y_1)-\piC(y_2)\| \Bigr\} $$
\end{prop}

\begin{proof}
    Let $p\in \C[X]^L_{\C[Y]}$ and let $I$ be the ideal of $\C[Y]\otimes_{\C} \C[Y]$ generated by the elements $(x_i\circ\pi) \otimes 1 - 1\otimes (x_i\circ\pi)$ where the $x_i$ are the coordinates of $X$. By definition, we have $p\otimes 1 - 1\otimes p \in \overline{I}$. So, by Lemma \ref{PropEclatementNormalise}, this is equivalent to having, locally on all $(Y\times Y)(\C)$, the inequality $|p(y_1)-p(y_2)|<C\sup_i |x_i\circ\pi(y_1)- x_i\circ\pi(y_2)| = C\|\pi(y_1)-\pi(y_2)\|$.
\end{proof}

\begin{rmq}
    In particular, if $p\in\C[Y]$ verifies such inequality around $(y_1;y_2)\in Y(\C)\times Y(\C)$, then $\pi(y_1)=\pi(y_2)$ implies $p(y_1)=p(y_2)$. Since the elements of the saturation $\widehat{\C[X]_{\C[Y]}}$ correspond to the elements of $\C[Y]$ which are constant on the fibers of $\piC$, this proposition is coherent with the inclusion $\C[X]^L_{\C[Y]}\subset \widehat{\C[X]_{\C[Y]}}$.
\end{rmq}

\begin{cor}
    Let $\pi : Y\to X$ be a dominant morphism of complex affine varieties and $\pi' : X'_Y \to X$ be the relative normalization morphism. Then
    $$ \C[X]^{L,+}_{\C[Y]} = \Bigl\{ p\in \C[X'_Y] \mid \text{locally }|p(y_1)-p(y_2)| \leqslant C\|\piC'(y_1)-\piC'(y_2)\| \Bigr\} $$
\end{cor}

\begin{ex}
    Let $X := \Spec(\C[x,y]/\langle y^4-x^5\rangle)$. Its normalization, which is the same as its seminormalization, is given by $\pi' : X' \to X$ with $X' := \Spec(\C[x,y]/\langle y^4-x\rangle)$ and $\piC' : (x,y) \mapsto (x,xy)$. It is obtained by adding the fraction $y/x$ in the coordinate ring of $X$. This fraction becomes $y$ in $\C[X']$, let us see that it is not an element of $\C[X^L]$. According to Proposition \ref{PropLipSatCritereGeo}, if $y$ corresponds to an element of $\C[X^L]$, we should have an inequality of the following form in a small neighborhood of $(0,0)$
    $$ |0-y| \leqslant C\|\piC'(0,0)-\piC'(x,y)\| $$
    i.e.
    $$ |y| \leqslant C\sup(|x|,|xy|) = C\sup(|y|^4,|y|^5) $$
    But there is no such inequality when $(x,y)$ goes to zero. This means that $y/x \in \C[X^+]\setminus \C[X^L]$. If we look now at the fraction $f = y^2/x \in \K(X)$, it becomes the element $xy^2 = y^6$ in $\C[X']$ and so it verifies the above inequality. In fact, for $(x_1,y_1),(x_2,y_2)\in X'(\C)$, we have 
    $$y_1^6- y_2^6 = (y_1+y_2)(y_1^5-y_2^5)-y_1y_2(y_1^4-y_2^4).$$
    So, for any point of $X'(\C)$, we can consider a compact neighborhood where we will have
    $$|y_1^6- y_2^6| \leqslant C\sup(|y_1^5-y_2^5|,|y_1^4-y_2^4|)$$
    We get that $y^2/x \in \C[X]^L_{\C[X']}$. In fact, by adding this fraction and $y^3/x^2$ to the coordinate ring of $\C[X]$, we obtain the Lipschitz saturation.
\end{ex}

\subsection{Locally Lipschitz rational functions}\label{SubsecLocLipRational}

We are now interested in rational functions of an algebraic variety that can be extended to the closed points of the variety as a locally Lipschitz function. The main goal of this section is to obtain a version of Pham-Teissier's \cite{PhamTeissier} Theorem 1.2 for the relative Lipschitz seminormalization of an algebraic variety in another. This will be given by Theorem \ref{TheoLipSatEstLipRatio}, which can also be seen as the Lipschitz version of \cite{BFMQ} Theorem 4.11. about regulous functions and seminormalization.

\begin{definition}
    Let $X$ be an affine variety over $\C$. We say that $f\in \K(X)$ is a locally Lipschitz rational function if $f$ extends to $X(\C)$ such that, for all $x_0\in X(\C)$, there exists $V \in \Vcal_{X(\C)}(x_0)$ and
    $$ \exists C>0, \forall x_1,x_2\in V, |f(x_1)-f(x_2)|\leqslant C\lVert x_1-x_2 \rVert_{\infty} $$
    We denote by $\KL(X(\C))$ the set of locally Lipschitz rational functions on $X$.
\end{definition}

\begin{rmq}
    A rational function $f\in\K(X)$ belongs to $\KL(X(\C))$ if and only if the function $|f\otimes1-1\otimes f|/\lVert x\otimes 1- 1\otimes x \rVert$ is locally bounded on $X(\C)\times X(\C)$ in the sense of definition \ref{DefLocBounded}.
\end{rmq}

\begin{rmq}
    We obviously have $\KL(X(\C))\subset \KO(X(\C))$ and so
    $$ \KL(X(\C)) \subset \C[X^+] \subset \C[X'] $$
    Therefore $\KL(X(\C))$ is a finite $\C$-algebra, so it corresponds to the coordinate ring of an affine variety. We will show that this variety is the Lipschitz saturation $X^L$.
\end{rmq}

\begin{ex}
    Let $X = \Spec(\C[x,y]/\langle xy(y-x)\rangle)$ and $f:X(\C) \to \C$ be the function defined by 
    $$ f(x,y) = \left\{ \begin{array}{cl}
     \frac{2xy}{x+y} & \text{ if } (x,y)\neq (0,0)  \\
     0 & \text{ else}
    \end{array} \right. .$$
    We have $f(x,y)=0$ if $xy=0$ and $f(x,y)=x$ if $y=x$. So, for $(x_1,y_1)$ and $(x_2,y_2) \in X(\C)$, we have:
    $$\left\{\begin{array}{ccl}
        x_1.y_1 = 0\text{ and }x_2.y_2 = 0&\text{ implies }&|f(x_1,y_1)-f(x_2,y_2)| = 0.\\
        x_1 = y_1\text{ and }x_2 = 0&\text{ implies }&|f(x_1,y_1)-f(x_2,y_2)| = |x_1|\text{ and }|x_1-x_2| = |x_1|.\\
        x_1 = y_1\text{ and }y_2 = 0&\text{ implies }&|f(x_1,y_1)-f(x_2,y_2)| = |y_1|\text{ and }|y_1-y_2| = |y_1|.\\
        x_1 = y_1\text{ and }x_2 = y_2&\text{ implies }&|f(x_1,y_1)-f(x_2,y_2)| = |x_1-x_2|.
    \end{array}\right.$$
    In every case, we get $|f(x_1,y_1)-f(x_2,y_2)| \leqslant \lVert (x_1,y_1)-(x_2,y_2) \rVert_{\infty}$ so $f\in \KL(X(\C))$.\\
\end{ex}

The inclusion $\KL(X(\C))\subset \KO(X(\C))$ allows us to get the following characterizations of locally Lipschitz rational functions, which are direct consequences of results obtained in \cite{Bernard2021}.

\begin{prop}
    Let $X$ be a complex affine variety. Then
    $$\KL(X(\C)) = \C[X]'_{\mathcal{L}(X(\C))}$$
    where $\mathcal{L}(X(\C))$ denotes the set of locally Lipschitz functions defined on $X(\C)$.
\end{prop}

\begin{proof}
    This is a corollary of \cite{Bernard2021} Corollary 4.18.
\end{proof}

The following proposition says that we are considering "algebraic" locally Lipschitz functions, according to the terminology of \cite{Jelonek2021AlgebraicBilipHomeo}.

\begin{prop}
    Let $X$ be a complex affine variety. Then, the ring $\KL(X(\C))$ corresponds to the set of locally Lipschitz functions with a Zariski-closed graph.
\end{prop}

\begin{proof}
    This is a corollary of \cite{Bernard2021} Corollary 4.19.
\end{proof}

In the two following propositions, we study the composition of locally Lipschitz functions by a finite morphism of algebraic varieties.

\begin{prop}\label{PropVarphiInjEtIm}
    Let $\pi : Y\to X$ be a finite morphism of affine complex varieties. Then the map 
    $$\begin{array}{cccc}
        \varphi : & \KL(X(\C)) & \to & \KL(Y(\C)) \\
         & f & \mapsto & f\circ\piC
    \end{array}$$ is injective and 
    $$ \im(\varphi) = \Bigl\{ g\in\KL(Y(\C)) \mid \text{locally } |g(x)-g(y)| \leqslant C\|\piC(x)-\piC(y)\| \Bigr\}.$$
\end{prop}

\begin{proof}
    The map $\varphi$ can be seen as the restriction to $\KL(X(\C))$ of the morphism defined in \cite{Bernard2021} Lemma 4.9. So if $\varphi$ is well-defined, it is injective and if we consider $f\in \KL(X(\C))$, then $f\circ\piC \in \KO(Y(\C))$. Moreover, since $f$ and $\piC$ are locally Lipschitz, then for every point of $Y(\C)\times Y(\C)$, we can find some constants $C,C'>0$ and a neighborhood where
    $$ |f\circ\piC(y_1)-f\circ\piC(y_2)| \leqslant C\|\piC(y_1)-\piC(y_2)\| \leqslant CC'\|y_1-y_2\|.$$
    So $f\circ\piC \in \KL(Y(\C))$ and $\varphi$ is well-defined. From the previous first inequality, we get the inclusion 
    $$ \im(\varphi) \subset \Bigl\{ g\in\KL(Y(\C)) \mid \text{locally } |g(x)-g(y)| \leqslant C\|\piC(x)-\piC(y)\| \Bigr\} .$$
    In order to show the other inclusion, let us consider $g\in \KL(Y(\C))$ such that $|g\otimes1-1\otimes g|/\|\piC\otimes 1-1\otimes\piC\|$ is locally bounded. We know, by \cite{Bernard2021} Proposition 4.11, that there exists a function $f\in\KO(X(\C))$ such that $f\circ\piC = g$. Since $\pi$ is finite, it is surjective and so, by Lemma \ref{LemLocBoundIfCircPIisLocBound}, the function $|f\otimes1-1\otimes f|/\|x\otimes 1-1\otimes x\|$ is locally bounded. Hence $f\in \KL(X(\C))$.
\end{proof}

In the case where $\varphi$ is an isomorphism of rings, the next proposition tells us what it means for $\piC$ to be a biLipschitz homeomorphism, as defined here.
 
\begin{definition}
    Let $(M,d)$ and $(M',d')$ be two metric spaces. A map $f:M\to M'$ is a biLipschitz homeomorphism if $f$ is a bijection and if there exists a real constant $C\geqslant 1$ such that for all $x,y\in M$
    $$ \frac{1}{C}.d(x,y) \leqslant d'(f(x),f(y)) \leqslant Cd(x,y) .$$
    We say that the map $f$ is a \textit{locally} biLipschitz homeomorphism if, for every point of $M$, there exists a neighborhood on which $f$ is a biLipschitz homeomorphism.
\end{definition}

\begin{rmq}
    A map $f:M\to M'$ is a biLipschitz homeomorphism if and only if $f$ is Lipschitz, bijective, and $f^{-1}$ is Lipschitz.
\end{rmq}

\newpage
\begin{prop}\label{PropVarphiIsoSSIBilip} 
    Let $\pi : Y\to X$ be a finite morphism of affine complex varieties. Then the following properties are equivalent:
    \begin{enumerate}
        \item The morphism $\varphi : \KL(X(\C))\to \KL(Y(\C))$ is an isomorphism of rings.
        \item The morphism $\piC$ is a locally biLipschitz homeomorphism.
    \end{enumerate}
\end{prop}

\begin{proof}
    Suppose that $\piC$ is a locally biLipschitz homeomorphism, and let $f\in \KL(Y(\C))$. Then, for all $x_0,y_0\in Y(\C)$, we have $C_1,C_2>0$ and two neighborhoods $V_1,V_2\in \Vcal_{(Y\times Y)(\C)}(x_0,y_0)$ such that 
    $$ \forall (x,y)\in V_1\text{\hspace{0.8cm} }\frac{1}{C_1}.\|x-y\| \leqslant \|\piC(x)-\piC(y)\| \leqslant C_1.\|x-y\| $$
    and $$ \forall (x,y)\in V_2\text{\hspace{0.8cm} }|f(x)-f(y)| \leqslant C_2.\|x-y\| .$$
    So, on $V_1\cap V_2$, we have 
    $$ |f(x)-f(y)| \leqslant C_2.\|x-y\| \leqslant C_2.C_1.\|\piC(x)-\piC(y)\| .$$
    By Proposition \ref{PropVarphiInjEtIm}, this means that $f\in \im(\varphi)$ and so $\varphi$ is an isomorphism. Conversely, suppose that $\varphi$ is an isomorphism, we want to show that $\piC$ is a biLipschitz homeomorphism. First, let us see that $\piC$ is bijective. We have $\C[X]\inj \KL(X(\C))\inj \KO(X(\C)) \simeq \C[X^+]$, so, by Proposition \ref{PropSousExtSubEstSub}, the extension $\C[X]\inj \KL(X(\C))$ is subintegral. We also have $\KL(X(\C))\simeq \KL(Y(\C)) \inj \C[Y^+]$, so the extension $\KL(X(\C))\inj \C[Y^+]$ is subintegral. Hence, we get $\C[X]\inj \C[Y^+]$ and the extension is subintegral. By Proposition \ref{PropSubEquiMemeFctRatioCont}, this means that $\KO(X(\C)) \simeq \KO(Y^+(\C))$. By Theorem \ref{TheoIsoSchemaSN}, we have, $\KO(Y^+(\C))\simeq \C[Y^+] \simeq \KO(Y(\C))$. Finally, we get $\KO(X(\C))\simeq \KO(Y(\C))$ and Proposition \ref{PropSubEquiMemeFctRatioCont}, tells us that $\piC$ is bijective. Now, to see that $\piC$ is locally biLipschitz, let us consider $p_i\in \C[Y]$ such that
    $$\begin{array}{cccc}
        p_i : & Y(\C) & \to & \C \\
         & (y_1,\dots,y_n) & \mapsto & y_i
    \end{array}.$$
    In particular, we have $p_i\in \KL(Y(\C))$ and since $\varphi$ is an isomorphism, we get $p_i\circ\piC^{-1} \in \KL(X(\C))$. Hence for all $i\in \{1,\dots,n\}$, for all $x_0,y_0\in X(\C)$, there exists $C_i>0$ and $V_i\in\Vcal_{X(\C)\times X(\C)}(x_0,y_0)$ such that 
    $$ \forall (x,y)\in V_i\text{\hspace{0.8cm} }|p_i\circ\piC^{-1}(x)-p_i\circ\piC^{-1}(y)| \leqslant C_i.\|x-y\|.$$
    So, on $\cap_i V_i$, we obtain
    $$ \|\piC^{-1}(x)-\piC^{-1}(y)\| = \sup_i|p_i\circ\piC^{-1}(x)-p_i\circ\piC^{-1}(y)| \leqslant \sup_i(C_i).\|x-y\| .$$
\end{proof}

We state now the main theorem of this section, which links the relative Lipschitz seminormalization of a variety in another with locally Lipschitz rational functions.

\begin{theorem}\label{TheoLipSatEstLipRatio}
    Let $\pi : Y\to X$ be a dominant morphism of affine varieties. Then 
    $$ \C[X^{L,+}_Y] \simeq \KL(X(\C))\times_{\KL(Y(\C))} \C[Y].$$
\end{theorem}

\begin{proof}
    Let us write 
    $$\xymatrix{
        Y \ar@{->}[r] & X'_Y \ar@{->}[r]^{\widetilde{\pi}_Y} \ar@/_1pc/[rr]_{\pi'_Y} & X^+_Y \ar@{->}[r]^{\pi^+_Y} & X
    }$$
    Let $f\in \C[X]^{L,+}_{\C[Y]}$, this means that $f\in \C[X]'_{\C[Y]}$ and $|f\otimes1-1\otimes f|/\|\pi'_Y\otimes 1-1\otimes \pi'_Y\|$ is locally bounded on $X'_Y(\C)$. By Lemmas \ref{LemLocBoundCircPIisLocBound} and \ref{LemLocBoundIfCircPIisLocBound}, this is equivalent to have $f\in\C[X]^+_{\C[Y]}$ and $|f\otimes1-1\otimes f|/\|\pi^+_Y\otimes 1-1\otimes \pi^+_Y\|$ locally bounded on $X^+_Y(\C)$ because $\widetilde{\pi_Y}$ is finite and 
    $$ \C[X]^{L,+}_{C[Y]} \subset \C[X]^+_{C[Y]} \subset \C[X]'_{C[Y]}. $$
    By \cite{BFMQ} Theorem 1.19, we have $\C[X^+_Y] = \KO(X(\C))\times_{\KO(Y(\C))} \C[Y]$. So $f\in \C[X]^{L,+}_{\C[Y]}$ if and only if there exists $g\in \KO(X(\C))$ such that $g\circ\pi^+_{Y(\C)} = f$ and ${|g\circ\pi^+_{Y(\C)}\otimes1-1\otimes g\circ\pi^+_{Y(\C)}|/\|\pi^+_{Y(\C)}\otimes 1-1\otimes \pi^+_{Y(\C)}\|}$ is locally bounded. Since $\pi^+_Y$ is finite, we can apply Lemmas \ref{LemLocBoundCircPIisLocBound} and \ref{LemLocBoundIfCircPIisLocBound} so the condition becomes "$|g\otimes1-1\otimes g|/\|x\otimes 1-1\otimes x\|$ is locally bounded". Then, this is exactly asking for $g$ to be in $\KL(X(\C))$. Hence $f\in \C[X]^{L,+}_{\C[Y]}$ if and only if $f\in \KL(X(\C))\times_{\KL(Y(\C))} \C[Y]$.
\end{proof}

If $X$ is a normal variety, then we get $\C[X] \inj \KL(X(\C))\to \KO(X(\C)) = \C[X]$ by \ref{PropLocalBoundIffIntegral}. Then, by taking $Y = X'$ in Theorem \ref{TheoLipSatEstLipRatio}, we get an algebraic version of Pham-Teissier's Theorem, which states that the locally Lipschitz rational functions correspond to the regular functions of the Lipschitz saturation.

\begin{cor}\label{CorLipSatEstLipRatio}
    Let $X$ be an affine complex variety and $\pi^L : X^L \to X$ be its Lipschitz saturation. The morphism
    $$\begin{array}{cccc}
        \varphi : & \KL(X(\C)) & \to & \C[X^L]\\
         & f & \mapsto & f\circ\piC^L
    \end{array}$$
    is an isomorphism.
\end{cor}

\subsection{For non-affine complex varieties}\label{Subnonaffine}

Let $X$ be an algebraic variety over $\C$. Consider a Zariski open set $U\subset X$, a finite number of regular functions $f_1,\dots,f_m$ on $U$ and a number $\varepsilon\in \R$. Then the sets of the form
$$ V(U;f_1,\dots,f_m,\varepsilon) := \Bigl\{ x\in U(\C) \mid |f_i(x)|<\varepsilon \text{ for }i=1,\dots,m \Bigr\} $$
form a basis for the open sets of a topology on $X(\C)$ called the \textit{strong topology}. If $X$ is affine, then $X(\C)\subset \C^n$, for some $n\in \mathbb{N}$ and the strong topology is induced by the Euclidean topology of $\C^n$. We can consider the sheaf $\KL_{X(\C)}$ such that $\KL_{X(\C)}(U(\C))$ is the set of locally Lipschitz rational functions, for any Z-open set $U$ of $X$. A dominant morphism $\pi:Y\to X$ between algebraic varieties induces an extension $\KL_{X(\C)} \to (\piC)_*\KL_{Y(\C)}$, hence a morphism $(Y(\C),\KL_{Y(\C)}) \to (X(\C),\KL_{X(\C)})$ of ringed spaces. 

Let $\pi : Y \to X$ be a dominant morphism between algebraic varieties over $k$, then we can define the $\Ocal_X$-module $(\Ocal_X)^{L}_Y$ such that $(\Ocal_X)^{L}_Y(U) = (\Ocal_X(U))^{L}_{\pi_*\Ocal_Y(U)}$ for each open set $U\subset X$. By \cite{Lipman1975} Proposition 3.1, remark (ii), we have that $(\Ocal_X)^L_Y$ is quasi-coherent and so, by \cite{EGAII} II Proposition 1.3.1, it corresponds to the structural sheaf of a variety $X^{L}_Y$. The relative normalization $X'_Y$ being well-defined (see \cite{BFMQ} Definition 1.8 for example), we can also consider the Lipschitz seminormalization $X^{L,+}_Y$.

Hence, all the previous results are valid for non-affine varieties.\\

In particular, by Corollary \ref{CorLipSatEstLipRatio}, we have an isomorphism of ringed spaces $(X(\C),\K_{X(\C)}^L) \simeq (X^L,\Ocal_{X^L})$. As a consequence, we can state that, as in the analytic case, the Lipschitz saturation determines a variety up to birational locally biLipschitz equivalence.

\begin{definition}
    Let $X$ and $Y$ be algebraic varieties over $\C$. We say that a map $h : Y(\C)\to X(\C)$ is birationally locally biLipschitz if it is a locally biLipschitz homeomorphism and if there exist two Z-dense Z-open $U_1 \subset Y(\C)$ and $U_2 \subset X(\C)$ such that $h_{\mid U_1} : U_1\to U_2$ is an isomorphism. In this situation, we say that $X$ and $Y$ are birationally locally biLipschitz equivalent.
\end{definition}

Note that two varieties can be birationally locally biLipschitz equivalent without admitting a morphism from one to another.

\begin{theorem}\label{TheoFinal}
    Let $X$ and $Y$ be two complex algebraic varieties. Then $X(\C)$ and $Y(\C)$ are birationally locally biLipschitz equivalent if and only if $X^L$ and $Y^L$ are isomorphic.
\end{theorem}

\begin{proof}
    Assume that $h : Y(\C) \to X(\C)$ is a locally biLipschitz and birational map. By \cite{BFMQ} Proposition 4.17, we get that $h$ is a homeomorphism for the Zariski topology and thus the ringed spaces $(Y(\C), \KL_{Y(\C)})$ and $(X(\C),\KL_{X(\C)})$ are isomorphic. By Corollary \ref{CorLipSatEstLipRatio}, we get that $X^L$ and $Y^L$ are isomorphic.
    Since $X^L(\C) \to X(\C)$ and $Y^L(\C) \to Y(\C)$ are birational locally biLipschitz morphisms, we get the converse implication.
\end{proof}

\begin{cor}
    Let $X$ and $Y$ be Lipschitz saturated algebraic varieties over $\C$. If $X(\C)$ and $Y(\C)$ are birationally locally biLipschitz equivalent, then they are isomorphic.
\end{cor}

\subsection{Locally biLipschitz algebraic morphisms}\label{SubsecLocLipAlgMorphisms}

The aim of this subsection is to give criteria, using the Lipschitz saturation, for two varieties to be linked by a locally biLipschitz algebraic morphism.

We start by stating the universal property of the Lipschitz seminormalization in terms of locally biLipschitz algebraic morphisms. In order to do so, we first need the simple following lemma concerning the idempotence of the Lipschitz saturation.

\begin{lemma}[\cite{FGST2020LipschitzAnAlgebraicApproch} Lemma 1.4.19]\label{LemIdempotent}
    Let $A\inj B$ be an extension of rings. Then 
    $$(A^L_B)^L_B = A^L_B.$$
\end{lemma}

\begin{proof}
    By definition, we have
    $$ A^L_B = \left\{ b\in B \mid b\otimes 1-1\otimes b \in \overline{I} \right\} \text{ and } (A^L_B)^L_B = \left\{ b\in B \mid b\otimes 1-1\otimes b \in \overline{J} \right\}$$
    where $I = \langle a\otimes 1-1\otimes a \mid a\in A\rangle$ and $J = \langle c\otimes 1-1\otimes c \mid c\in A^L_B\rangle$. By definition of $A^L_B$, the elements generating $J$ are the elements of $\overline{I}$. So $J = \overline{I}$ and $\overline{J} = \overline{\overline{I}} = \overline{I}$ by \cite{LejeuneTeissierCloture} Corollary 1.8.
\end{proof}

\begin{rmq}
    As a consequence, we get $(X^L(\C),\KL_{X^L(\C)}) \simeq (X(\C),\KL_{X(\C)})$ and so, by Proposition \ref{PropVarphiIsoSSIBilip}, we have that the Lipschitz saturation morphism $\piC^L$ is a locally biLipschitz homeomorphism.
\end{rmq}

For a finite morphism $Y\to X$, we can state the universal property of the relative Lipschitz saturation of $X$ in $Y$. It is the biggest variety, between $X$ and $Y$, which is linked to $X$ by a finite, birational morphism whose restriction to the closed points is a locally biLipschitz homeomorphism.

\begin{theorem}[Universal property of the relative Lipschitz saturation]\label{TheoPULipSatu}
    Let $Y \to X$ be a finite morphism of complex algebraic varieties. Then $X^L_Y$ is the unique variety with the following property: For every intermediate variety $Y \to Z\to X$ with $\piC : Z(\C) \to X(\C)$ locally biLipschitz, there exists a unique morphism $\pi_Z^L : X_Y^L \to Z$ such that the following diagram commutes.
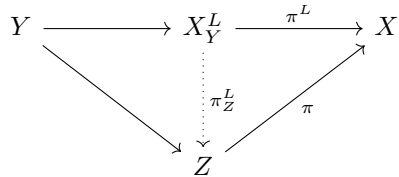
\[\begin{tikzcd}
	Y && {X^L_Y} && X \\
	\\
	&& Z
	\arrow[from=1-1, to=3-3]
	\arrow["\pi"', from=3-3, to=1-5]
	\arrow[from=1-1, to=1-3]
	\arrow["{\pi^L}", from=1-3, to=1-5]
	\arrow["{\pi^L_Z}", dotted, from=1-3, to=3-3]
\end{tikzcd}\]

    Moreover $\pi_{Z(\C)}^L$ is a locally biLipschitz homeomorphism.
\end{theorem}

\begin{proof}
We can suppose the varieties to be affine. Since $\piC : Z(\C) \to X(\C)$ is a locally biLipschitz homeomorphism, we have $\KL(Z(\C)) \simeq \KL(X(\C))$ by Proposition \ref{PropVarphiIsoSSIBilip}. So we get the following commutative diagram:
    $$
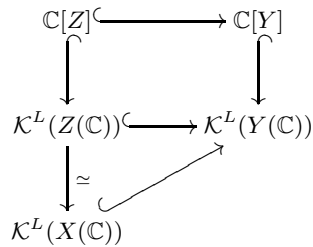
\xymatrix{
        \C[Z] \ar@{^{(}->}[r] \ar@{^{(}->}[d] & \C[Y] \ar@{^{(}->}[d]\\
        \KL(Z(\C)) \ar@{^{(}->}[r] \ar@{->}[d]^{\simeq} & \KL(Y(\C))\\
        \KL(X(\C)) \ar@{^{(}->}[ur] &
    }$$
    By Theorem \ref{TheoLipSatEstLipRatio}, the ring $\C[X^L_Y]$ is the fiber product of $\KL(X(\C))$ and $\C[Y]$ above $\KL(Y(\C))$, so we get the inclusion $\C[Z] \inj \C[X^L_Y]$ which induces a morphism $\pi^L_{Z} : X^L_Y \to Z$. By the remark following Lemma \ref{LemIdempotent}, we get $\KL(X(\C)) \inj \KL(Z(\C)) \inj \KL(X^L(\C)) \simeq \KL(X(\C))$. So $\pi^L_{Z(\C)}$ is a locally biLipschitz homeomorphism.
\end{proof}

Thanks to the following lemma, we get a generalization of this result: the universal property of the Lipschitz seminormalization.

\begin{lemma}[\cite{BFMQ} Proposition 5.6]\label{LemPIhomeoDoncPIfini}
    Let $\pi : Y\to X$ be a morphism of complex algebraic varieties. If $\piC$ is a homeomorphism for the Euclidean topology, then $\pi$ is finite.
\end{lemma}

\begin{theorem}[Universal property of the relative Lipschitz seminormalization]\label{TheoPULipSemi}
    Let $Y \to X$ be a dominant morphism of complex algebraic varieties. Then $X^{L,+}_Y$ is the unique variety with the following property: For every intermediate variety $Y \to Z\to X$ with $\piC : Z(\C) \to X(\C)$ locally biLipschitz, there exists a unique morphism $\pi_Z^{L,+} : X_Y^{L,+} \to Z$ such that the following diagram commutes.

\[\begin{tikzcd}[sep=large]
	Y & {X^{L,+}_Y} & X \\
	& Z
	\arrow[from=1-1, to=1-2]
	\arrow[from=1-1, to=2-2]
	\arrow["{\pi^{L,+}}", from=1-2, to=1-3]
	\arrow["{\pi^{L,+}_Z}", dashed, from=1-2, to=2-2]
	\arrow["\pi"', from=2-2, to=1-3]
\end{tikzcd}\]
    Moreover $\pi_{Z(\C)}^{L,+}$ is a locally biLipschitz homeomorphism.
\end{theorem}

\begin{proof}
    Let $Y \to Z\to X$ with $\piC : Z(\C) \to X(\C)$ locally biLipschitz. By Lemma \ref{LemPIhomeoDoncPIfini}, the morphism $\piC : Z \to X$ is finite. This implies that $\Ocal_X\inj \Ocal_Z$ is integral, and so $\Ocal_Z\inj (\Ocal_X)'_{\Ocal_Y} = \Ocal_{X'_Y}$. Since, by definition, we have $X^{L,+}_Y = X^L_{X'_Y}$, we conclude the proof by applying Theorem \ref{TheoPULipSatu} to $X'_Y\to Z \to X$.
\end{proof}

By taking $Y=X'$, we get the universal property of the classical Lipschitz saturation for an algebraic variety.

\begin{cor}\label{CorPULipSatu}
    Let $X$ be a complex algebraic variety. Then $X^L$ is the biggest variety admitting a morphism $\pi^L : X^L \to X$ such that $\pi^L$ is birational and $\piC^L$ is locally biLipschitz.
\end{cor}

For irreducible varieties, we can drop the assumption that $\pi^L$ is birational thanks to the following lemma. 

\begin{lemma}[\cite{BFMQ} Proposition 5.3]\label{LemPIbijDoncPIbiration}
    Let $X$ and $Y$ be irreducible complex varieties. Then a morphism from $Y$ to $X$ inducing a bijection on the closed points is quasi-finite and birational.
\end{lemma}

\begin{theorem}[Universal property of the Lipschitz saturation for irreducible varieties]
    Let $X$ be an irreducible complex algebraic variety. Then $X^L$ is the unique variety with the following property: For every morphism $\pi : Z\to X$ such that $\piC$ is a locally biLipschitz homeomorphism, there exists a unique morphism $\pi^L_Z : X^L \to Z$ such that the following diagram commutes.
    $$\xymatrix{
        X^L \ar@{.>}[d]^{\pi^L_Z} \ar[rr]^{\pi^L} && X \\
        Z \ar[urr]_{\pi} &&
    }$$
    Moreover $\pi_{Z(\C)}^L$ is a locally biLipschitz homeomorphism.
\end{theorem}

\begin{proof}
    By lemma \ref{LemPIhomeoDoncPIfini}, such a morphism $\pi : Z\to X$ is finite and bijective, so it is a Zariski homeomorphism. This implies that $Z$ is also irreducible. Then, we can apply Lemma \ref{LemPIbijDoncPIbiration} and so the morphism is birational. We get the theorem by applying Corollary \ref{CorPULipSatu}.
\end{proof}

For a given morphism of complex algebraic varieties such that its restriction to the closed points is a locally biLipschitz homeomorphism, we can deduce from the universal properties stated before, algebraic conditions for the morphism to be an isomorphism.\\

\begin{cor}\label{CorPIisoIFF}
    Let $\pi : Y\to X$ be a morphism of complex algebraic varieties such that $\piC$ is a locally biLipschitz homeomorphism. Then $\pi$ is an isomorphism if and only if $X$ is Lipschitz saturated in $Y$.
\end{cor}

\begin{proof}
    By lemma \ref{LemPIhomeoDoncPIfini}, we get that $\pi$ is finite. So we can apply the universal property of the relative Lipschitz saturation and we get $X^L_Y = Y$. Hence $X \simeq X^L_Y$ if and only if $\pi$ is an isomorphism.
\end{proof}

\begin{cor}
    Let $\pi : Y\to X$ be a morphism of complex algebraic varieties such that $X$ is Lipschitz saturated. Then $\piC$ is a locally biLipschitz homeomorphism if and only if $\pi$ is an isomorphism.
\end{cor}

\begin{proof}
    If $\piC$ is an homeomorphism for the strong topology, then it is finite by Lemma \ref{LemPIhomeoDoncPIfini}. Then, by Theorem \ref{TheoLipSatEstLipRatio}, we have 
    $$\Ocal_X \inj \Ocal_{X^L_Y} \inj \KL_{X(\C)} .$$
    Since $X$ is Lipschitz saturated, we get $\Ocal_X \simeq \KL_{X(\C)}$ and so $\Ocal_X \simeq \Ocal_{X^L_Y}$. So $X$ is Lipschitz saturated in $Y$ and we conclude by Corollary \ref{CorPIisoIFF}.
\end{proof}

\bibliographystyle{abbrv}
\bibliography{Biblio.bib}

\noindent François Bernard, Université Paris-Cité, IMJ-PRG\\
Bâtiment Sophie Germain, Place Aurélie Nemours, 75013, Paris, France\\
\textit{E-mail address: \textbf{fbernard@imj-prg.fr}}

\end{document}